\documentclass[a4paper]{amsart}

\usepackage{amsmath,amssymb,amsthm}
\usepackage{enumitem}
\usepackage{mathtools}
\usepackage[all]{xy}
\usepackage{tikz}

\usepackage{longtable}

\theoremstyle{plain}
\newtheorem{theorem}{Theorem}[section]
\newtheorem{proposition}[theorem]{Proposition}

\newtheorem{corollary}[theorem]{Corollary}

\theoremstyle{definition}
\newtheorem{definition}[theorem]{Definition}
\newtheorem*{acknowledgements}{Acknowledgements}

\theoremstyle{remark}

\newtheorem{remark}[theorem]{Remark}
\newtheorem{notation}[theorem]{Notation}
\newtheorem{convention}[theorem]{Convention}

\newtheorem{condition}[theorem]{Condition}

\numberwithin{equation}{theorem}

\DeclareMathOperator{\Pic}{Pic}

\DeclareMathOperator{\pr}{pr}

\DeclareMathOperator{\Aut}{Aut}
\DeclareMathOperator{\Bl}{Bl}

\DeclareMathOperator{\sub}{sub}

\DeclareMathOperator{\gen}{gen}
\DeclareMathOperator{\spe}{sp}

\DeclareMathOperator{\DF}{DF}
\DeclareMathOperator{\ord}{ord}
\DeclareMathOperator{\vol}{vol}

\newcommand{\sF}{\mathcal{F}}

\newcommand{\sL}{\mathcal{L}}

\newcommand{\sC}{\mathcal{C}}

\newcommand{\sX}{\mathcal{X}}

\newcommand{\bC}{\mathbf{C}}

\newcommand{\bG}{\mathbf{G}}

\newcommand{\bP}{\mathbf{P}}
\newcommand{\bQ}{\mathbf{Q}}
\newcommand{\bR}{\mathbf{R}}

\newcommand{\bA}{\mathbf{A}}

\DeclareMathOperator{\Pas}{\mathcal{P}}
\newcommand{\PF}{\Pas_{F_4}}
\newcommand{\PG}{\Pas_{A_1 \times G_2}}
\newcommand{\tX}{\widetilde X}
\newcommand{\blank}{{-}}

\newcommand{\dF}{
{
\SelectTips{}{12}
\objectmargin={0pt}
\objectheight={20pt}
\objectwidth={5pt}
\xygraph{!{<0cm,0cm>;<0.7cm,0cm>:<0cm,0.7cm>::}
\circ*=!D{\scriptstyle 1}-[r] \circ*=!D{\scriptstyle 2} -@2{-}|(.6)@{>}[r] \circ*=!D{\scriptstyle 3}-[r] \circ*=!D{\scriptstyle 4} 
}}
}
\newcommand{\dG}{
{
\SelectTips{}{12}
\objectmargin={0pt}
\objectheight={20pt}
\objectwidth={5pt}
\xygraph{!{<0cm,0cm>;<0.7cm,0cm>:<0cm,0.7cm>::}
\circ*=!D{\scriptstyle 1}-@3{-}|(.5)@{<}[r] \circ*=!D{\scriptstyle 2}
}}
}
\newcommand{\dAG}{
{
\SelectTips{}{12}
\objectmargin={0pt}
\objectheight={20pt}
\objectwidth={5pt}
\xygraph{!{<0cm,0cm>;<0.7cm,0cm>:<0cm,0.7cm>::}
\circ*=!D{\scriptstyle 0}-@0{-}[r] \circ*=!D{\scriptstyle 1}-@3{-}|(.5)@{<}[r] \circ*=!D{\scriptstyle 2}
}}
}

\newcommand{\rG}{
\begin{tikzpicture}
\draw[->] (0,0)--(6*60:1)node[right] {$\alpha_1$};
\draw[->] (0,0)--(60:1);
\draw[->] (0,0)--(2*60:1);
\draw[->] (0,0)--(3*60:1);
\draw[->] (0,0)--(4*60:1);
\draw[->] (0,0)--(5*60:1);
\draw[->] (0,0)--(30:{sqrt(3)});
\draw[->] (0,0)--(30+60:{sqrt(3)});
\draw[->] (0,0)--(30+2*60:{sqrt(3)})node[left] {$\alpha_2$};
\draw[->] (0,0)--(30+3*60:{sqrt(3)});
\draw[->] (0,0)--(30+4*60:{sqrt(3)});
\draw[->] (0,0)--(30+5*60:{sqrt(3)});
\end{tikzpicture}
}

\AtBeginDocument{%
\def\MR#1{}
}

\setenumerate{
label=(\arabic*),
font=\normalfont
}

\title[K\"ahler-Einstein metrics on Pasquier's two-orbits varieties]{K\"ahler-Einstein metrics on\\Pasquier's two-orbits varieties}
\author[A. KANEMITSU]{Akihiro KANEMITSU}
\date{\today}
\address{Department of Mathematics, Graduate school of Science, Kyoto University, Kyoto 606-8502, Japan}
\email{kanemitu@math.kyoto-u.ac.jp}
\thanks{The author is a JSPS Research Fellow and supported by the Grant-in-Aid for JSPS fellows (JSPS KAKENHI Grant Number 18J00681).}


\begin{document}

\begin{abstract}
We show that there exist K\"ahler-Einstein metrics on two exceptional Pasquier's two-orbits varieties.
As an application, we will provide a new example of K-unstable Fano manifold with Picard number one.
\end{abstract}

\maketitle

\section*{Introduction}
\subsection{Pasquier's two-orbits varieties}
In 2009, as an application of his study on horospherical varieties, Pasquier
classified a special kind of two-orbits varieties that satisfies the following condition \cite{Pas09}:

\begin{condition}\label{cond:Pasquier}
\hfill
\begin{itemize}
 \item  $X$ is a Fano manifold with Picard number one.
 \item The connected automorphism group $G \coloneqq \Aut^{0} (X)$ acts on $X$ with two orbits $X^0 \coprod Z$, where $X^0$ is the open orbit and $Z$ is the closed orbit.
 \item The blow-up $\Bl_Z X$ is again a two-orbits variety $X^0 \coprod E$ with respect to $G$, where $E$ is the exceptional divisor of the blow-up.
\end{itemize}
\end{condition}

In summary, he showed that each Fano manifold with the above condition satisfies one of the following:
\begin{enumerate}
 \item $X$ is a horospherical variety; \label{item:Pas1}
 \item $X$ is isomorphic to a $23$-dimensional Fano manifold $\PF$ with $G = F_4$;
 \item $X$ is isomorphic to an $8$-dimensional Fano manifold $\PG$ with $G=A_1 \times G_2$.
\end{enumerate}
In the course of his study, he also determined the group $G$ for case \ref{item:Pas1}; in each of horospherical cases, $G$ is not reductive, and hence there are no K\"ahler-Einstein metrics on $X$ \cite{Mat57}.

In a previous article \cite{Kan20}, the author studied stability of tangent bundles for varieties satisfying Condition~\ref{cond:Pasquier}.
Therein we have provided examples of Fano manifolds with Picard number one whose tangent bundles are unstable, which disproved a folklore conjecture.
The present article is a continuation of this previous paper \cite{Kan20};
the purpose of this article is to show the existence of K\"ahler-Einstein metrics in the two remaining  cases $\PF$ and $\PG$:

\begin{theorem}\label{thm:KE}
 There are K\"ahler-Einstein metrics on $\PF$ and $\PG$.
\end{theorem}

\begin{remark}
In \cite[Theorem~4.1]{Del20YTD}, Delcriox also obtained the proof of the above result independently.
The method used in his proof is different from our proof; Delcriox's proof relies on a criterion of K-stability for spherical varieties, while our proof does not use the fact $\PF$ and $\PG$ are spherical, but relies on the $G$-equivariant valuative criterion of K-stability.
\end{remark}

\begin{remark}
More precisely, we will show that  $\PF$ and $\PG$ are K-polystable, which is equivalent to the existence of K\"ahler-Einstein metrics by virtue of \cite{Tia97,Don02,Ber16,CDS15a,CDS15b,CDS15c,Tia15}.
\end{remark}

The variety $\PG$ is a Mukai $8$-fold of genus $7$ (see \cite[Section~4.2]{BFM20} or \cite[Remark 4.2]{Kan20}).
Recall that, by \cite[Section~6]{Kuz18}, there are two isomorphic classes $X_{\spe}$ and $X_{\gen}$ of Mukai $8$-folds of genus $7$, and $X_{\gen}$ degenerates to $X_{\spe}$.
One can prove that $\PG \simeq X_{\gen}$ (cf. \cite[Remark 4.2]{Kan20}).
See Section~\ref{sect:remark} for a proof of this fact and also \cite[Proposition 4.8, Remark 4.10, Remark 5.4]{BFM20} for proofs.
Thus, by the general properties of K-moduli, we have:

\begin{corollary}\label{cor:K-unstable}
$X_{\spe}$ is not K-semistable.
In particular, there are no K\"ahelr-Einstein metrics on $X_{\spe}$.
\end{corollary}

\begin{remark}
The above corollary provides a new counter-example of a conjecture due to Odaka and Okada \cite[Conjecture 5.1]{OO13}.
See \cite{Fuj17} and \cite{Del20} for other examples.
\end{remark}

\subsection{Organization of the paper}
The article is organized as follows:
Section~\ref{sect:prelim} provides preliminaries on $G$-equivariant valuative criterion of K-stability.
The main ingredient of our proof is Theorem~\ref{thm:valuative_criterion}, which provides a criterion of K-stability by means of valuations.
In Subsection~\ref{subsect:outline}, we provide the outline of the proof.

In Section~\ref{sect:Pas}, we briefly recall the geometry of $\PF$ and $\PG$.
Then, in Section~\ref{sect:proof}, we completes our proof of Theorem~\ref{thm:KE}.
In the last section (=Section~\ref{sect:remark}), we prove Corollary~\ref{cor:K-unstable}.
We also give a remark on the relation between K\"ahelr-Einstein metrics and foliations.

\begin{convention}
 We work over the complex number field $\bC$.
 For a vector space $V$, $\bP_{\sub}(V)$ denotes the parameter space of 1-dimensional subspaces in $V$.
\end{convention}

\begin{acknowledgements}
The author wishes to express his gratitude to Professor Thibaut Delcroix for careful reading of the first draft of this paper, and also for discussions about the difference between his proof and our proof.
The author is also grateful to Professors Kento Fujita and Yuji Odaka for helpful discussions and also for answering questions on K-stability of Fano varieties.
The author also would like to thank Professor Baohua Fu for drawing his attention to \cite{BFM20} and for explaining Remark~\ref{rem:nonreductive}.
\end{acknowledgements}

\section{Preliminaries: $G$-equivariant valuative criterion of K-stability}\label{sect:prelim}
Here we briefly recall the definition of K-stability and its criterion.
For simplicity, we assume that $X$ is a smooth Fano variety and the polarization is given by $L \coloneqq -K_X$.

\subsection{K-stability}
K-stability is defined by using test configurations and their Donaldson-Futaki invariants, which encode informations about degenerations of the variety $X$ in question.
We briefly recall the definitions.

\begin{definition}[\cite{Tia97}, \cite{Don02}]
\hfill
\begin{enumerate}
\item
A test configuration of the pair $(X,L)$ is the following data:
\begin{itemize}
 \item a normal variety $\sX$;
 \item a proper flat morphism $p \colon \sX \to \bA^1$;
 \item a $p$-ample $\bQ$-line bundle $\sL$ on $\sX$;
 \item a $\bG_m$-action on $(\sX,\sL)$ which makes  the morphism $p$ $\bG_m$-equivariant with respect to the natural action $\bG_m$ on $\bA^1$ ($(g,x)\mapsto gx$);
 \item a $\bG_m$-equivariant isomorphism $(\sX\setminus \sX_0, \sL|_{\sX\setminus \sX_0}) \simeq (X\times(\bA^1\setminus\{0\},\pr_1^{*}L))$.
\end{itemize}
 \item A test configuration is called a \emph{product test configuration} if $(\sX,\sL) \simeq (X\times\bA^1, \pr_1^*L)$.
\end{enumerate}
\end{definition}

Let $(\overline{\sX},\overline{\sL})/\bP^1$ be the compactification of $(\sX,\sL)/\bA^1$ obtained by gluing $(\sX,\sL)/\bA^1$ and $(X\times \bP^1\setminus{0}, \pr^*_{1}L)$.
We denote by $\overline{p} \colon \sX \to \bP^1$ the natural projection.

\begin{definition}[\cite{Tia97}, \cite{Don02}, \cite{Wan12}, \cite{Oda13}]
For a test configuration of $(X,L)$, the \emph{Donaldson-Futaki} invariant is defined as follows:
\[
\DF(\sX,\sL) \coloneqq \frac{n}{n+1}\cdot \frac{\overline{\sL}^{n+1}}{L^n} +\frac{ \overline{\sL}^n\cdot K_{\overline{\sX}/\bP^1}}{L^n},
\]
 where $K_{\overline{\sX}/\bP^1} \coloneqq K_{\overline{\sX}}-\overline{p}^*K_{\bP^1}$.
\end{definition}

\begin{definition}
\hfill
\begin{itemize}
 \item The pair  $(X,L)$ is called K-semistable if $DF(\sX,\sL)\geq 0$ for any test configuration.
 \item The pair  $(X,L)$ is called K-polystable if it is K-semistable, and moreover $DF(\sX,\sL) = 0$ only when $(\sX,\sL)$ is a product test configuration.
\end{itemize}
\end{definition}

\subsection{Valuative criterion}
Here we recall the ($G$-equivariant) valuative criterion of K-stability \cite{Fuj19a}, \cite{Li17}, \cite{Zhu20}.
A prime divisor \emph{over} $X$ is a prime divisor $F$ on a resolution $\sigma \colon Y \to X$.

\begin{definition}
\hfill
\begin{enumerate}
 \item The log discrepancy of $X$ along $F$ is defined as follows:
 \[
 A(F) \coloneqq 1 + \ord _F (K_Y - \sigma^* K_X).
 \]
 \item (\cite{BJ20}) $S(F)$ is defined as follows:
 \[
 S(F) \coloneqq \frac{1}{L^n}\int_{0}^{\infty}\vol (L-xF) dx,
 \]
 where $\vol (L-xF)\coloneqq \vol_{Y}(\sigma^*L-xF)$.
 \item (\cite{Fuj19a}, \cite{Li17}) Then the $\beta$-invariant of $F$ is defined as follows:
  \[
  \beta(F) \coloneqq (L^n)(A(F)-S(F)).
  \]
\end{enumerate}
\end{definition}

\begin{theorem}[{\cite[Corollary 4.14]{Zhu20}}]\label{thm:valuative_criterion}
Assume that a reductive algebraic group $G$ acts on $(X,L)$.
If $\beta (F) > 0$ for any $G$-invariant prime divisor $F$ over $X$, then $X$ is K-polystable.
\end{theorem}

\subsection{$\xi$-invariant}
By using formula~\eqref{eq:xi} below, we replace the calculation of $\beta$-invariant with that of $\xi$-invariant \cite{Fuj15}, which corresponds to the slope stability of a subvariety in the sense of Ross-Thomas \cite{RT07}.

\begin{definition}
 Let $X$ be a Fano manifold and $Z$ a smooth subvariety of codimension $r$.
 Denote by $\varphi \colon \tX \to X$ the blow-up of $X$ along $Z$.
 Then the Seshadri constant of $\epsilon (Z)$ is defined as follows:
 \[
 \epsilon(Z) \coloneqq\max\{t\in \bR_{>0} \mid \text{$\varphi^*L-tE$ is nef on $\tX$}\}.
 \]
\end{definition}

\begin{proposition}[{\cite[Proposition~3.2]{Fuj15}}]\label{prop:Fujita}
 Set 
 \[
 \xi(Z) \coloneqq r \vol_X(-K_X) + (\epsilon(Z)-r)\vol_{\tX}(\varphi^*(-K_X)-\epsilon(Z)E)
 -\int_0^{\epsilon(Z)}\vol_{\tX}(\varphi^*(-K_X)-xE)dx.
 \]
 Then
 \[
 \xi(Z) = n\int_0^{\epsilon(Z)} (r-x)(E \cdot (\varphi^*(-K_X)-xE)^{n-1})dx.
 \]
\end{proposition}

\begin{remark}
 Note that, if $\varphi^*(-K_X) -xE$ is not big for $x\geq \epsilon (Z)$, then
 \[
 \vol_{\tX}(\varphi^*(-K_X)-xE) = 0
 \]
 for  $x \geq \epsilon (Z)$.
Thus
\begin{align}\label{eq:xi}
\begin{aligned}
  \xi (Z) &=  r \vol_X(-K_X)  -\int_0^{\epsilon(Z)}\vol_{\tX}(\varphi^*(-K_X)-xE)dx \\
 &=  r \vol_X(-K_X) -\int_0^{\infty}\vol_{\tX}(\varphi^*(-K_X)-xE)dx \\
 &= (A(E)-S(E))L^n \\
 &= \beta(E),
\end{aligned}
\end{align}
where $E$ is the exceptional divisor of the blow-up $\tX \to X$.
\end{remark}

\subsection{Outline of the proof}\label{subsect:outline}
Here we briefly sketch the outline of our proof.
In the case of Pasquier's two-orbits varieties $X = \PF$ or $\PG$,  the blow-up $\tX$ of $X$ along $Z$ admits a smooth fibration $\pi \colon \tX \to Y$ and hence we have the following diagram (see \cite[Propositions~2.5 and 2.7]{Kan20}):
\begin{equation}\label{eq:diagram}
\begin{gathered}
 \xymatrix{
 E \ar@{_{(}->}[d]  \ar[r]^-{\varphi|_E} & Z \ar@{_{(}->}[d] \\
 \tX \ar[d]_-{\pi} \ar[r]^-{\varphi} & X \\
 Y. &
 }
\end{gathered}
\end{equation}

Note that $\xi (Z) = \beta (E)$ by \eqref{eq:xi}.
By Condition~\ref{cond:Pasquier}, $E$ is the unique $G$-invariant prime divisor over $X$.
Note also that $\Aut^0(X)$ is reductive.
Thus 
\[
\text{$X$ is K-polystable} \iff \beta(E) > 0 \iff \xi(Z) >0.
\]
Thus it is enough to show $\xi(Z) >0$.
By virtue of Proposition~\ref{prop:Fujita}, we can calculate $\xi(Z)$ by computing intersection numbers on a homogeneous space $E$, and we will show $\xi(Z)>0$, which completes the proof.

\section{Geometry of Pasquier's two-orbits varieties}\label{sect:Pas}

Here we recall brief descriptions of Pasquier's two varieties $\PF$ and $\PG$.
For detailed descriptions and the original descriptions, we refer the reader to \cite{Kan20}, \cite{Pas09}.

\begin{proposition}[{\cite[Proposition~2.5 and 2.7]{Kan20}}]
Consider diagram~\eqref{eq:diagram}.
Let $H_X$ and $H_Y$ be the ample generator of $\Pic(X)$ and $\Pic(Y)$ respectively.
\begin{enumerate}
 \item If $X =\PF$, then we have the following:
 \begin{itemize}
  \item $\dim X = 23$ and $\dim Y = 20$;
  \item $-K_X = 8 H_X$;
  \item $E = -\pi^*H_Y + \varphi^*H_X$.
 \end{itemize}
 In particular, $\epsilon(Z) = 8$ and
 \[
 \xi(Z) = 23\int_0^{8}(3-x)(E\cdot (\varphi^*(-K_X)-xE)^{22})dx.
 \]
 
 \item If $X =\PG$, then we have the following:
 \begin{itemize}
  \item $\dim X = 8$ and $\dim Y = 5$;
  \item $-K_X = 6 H_X$;
  \item $E = -\pi^*H_Y + 2\varphi^*H_X$.
 \end{itemize}
 In particular, $\epsilon(Z) = 3$ and
 \[
 \xi(Z) = 8\int_0^{3}(2-x)(E\cdot (\varphi^*(-K_X)-xE)^{7})dx.
 \]
\end{enumerate}
\end{proposition}

\subsection{Geometry of the exceptional divisor}
Here we recall the description of the exceptional divisor $E$.
In the following, we denote by $p_Y \colon E \to Y$ and $p_Z \colon E \to Z$  the natural projections respectively.

\begin{definition}[{\cite[Definition~2.1]{Kan20}}]
The associated triples $(D,\omega_Y,\omega_Z)$ for $\PF$ and $\PG$ are defined as follows:
\begin{itemize}
 \item $(D,\omega_Y,\omega_Z) \coloneqq (F_4,\omega_1,\omega_3)$;
 \item $(D,\omega_Y,\omega_Z) \coloneqq (A_1\times G_2, \omega_2,\omega_0 + \omega_1)$.
\end{itemize}
Here $(D,\omega_Y,\omega_Z)$ consists of a Dynkin diagram $D$ and weights $\omega_Y$ and $\omega_Z$.
For the descriptions of the root systems $F_4$ and $A_1 \times G_2$, see Section~\ref{sect:proof} below.
\end{definition}

\begin{notation}
In the following $G$ is the simply connected semisimple algebraic group whose Dynkin diagram is $D$.
\begin{itemize}
 \item $V_Y$ and $V_Z$ are the irreducible $G$-representations with highest weights $\omega_Y$ and $\omega_Z$ respectively.
 \item $v_Y$ and $v_Z$ are the corresponding highest weight vectors.
 \item $[v_Y]\in \bP_{\sub}(V_Y)$ and $[v_Z]\in\bP_{\sub}(V_Z)$ are the corresponding points.
 \item $P_Y$ and $P_Z$ are the stabilizer groups of $[v_Y]$ and $[v_Z]$ respectively, which are parabolic subgroups.
 \item $P_{Y,Z} \coloneqq P_Y \cap P_Z$.
\end{itemize}
Then $\bP_{\sub}(V_Y) \supset G\cdot [v_Y] = G/P_Y$ and $\bP_{\sub}(V_Z) \supset G\cdot [v_Z] = G/P_Z$.
The polarizations of these two homogeneous spaces are given by the divisors corresponding to the weights $\omega_Y$ and $\omega_Z$.
\end{notation}

\begin{proposition}[{\cite[Sections 2.3, 2.4]{Kan20}}]
Then $E \simeq G/P_{Y,Z}$, $Y \simeq G/P_{Y}$ and $Z \simeq G/P_{Z}$.
The morphisms $p_Y \colon E \to Y$ and $p_Z \colon E \to Z$ are the quotient maps.

Moreover $H_Y$ is the divisor corresponding to $\omega_Y$ and $H_X|_Z$ is the divisor corresponding to $\omega_Z$.
\end{proposition}

\begin{corollary}
 \hfill
\begin{enumerate}
 \item If $X =\PF$, then
 \[
 \xi(Z) = 23 \int_0^{8}(3-x)(xp_Y^* H_Y + (8-x)p_Z^*(H_X|_Z))^{22}
 \]
 \item If $X =\PG$, then
 \[
 \xi(Z) = 8 \int_0^{3}(2-x)(xp_Y^* H_Y + (6-2x)p_Z^*(H_X|_Z))^{7}.
 \]
\end{enumerate}
Moreover $p_Y^* H_Y$ and $p_Z^* (H_X|_Z)$ are the divisors on $E = G/P_{Y,Z}$ corresponding to the weights $\omega_Y$ and $\omega_Z$ respectively.
\end{corollary}

\section{Proof of Theorem~\ref{thm:KE}}\label{sect:proof}

\subsection{Proof of Theorem~\ref{thm:KE} for $\PF$}
First we briefly recall the description of the root system $F_4$.
We follow the description in \cite{Bou02}.

Let $V = \bR^4$ be the Euclidian space of dimension $4$ and $e_i$ be the $i$-th fundamental vector.
Then the root system $F_4$ consists of the following $48$ roots:
\begin{itemize}
 \item $\pm e_i$ ($1 \leq i \leq 4$);
 \item $\pm e_i \pm e_j$ ($1 \leq i \leq j \leq 4$);
 \item $\frac{1}{2}(\pm e_1 \pm e_2 \pm e_3 \pm e_4)$.
\end{itemize}

Then the following four roots define a basis $\Delta$ of the root system:
\begin{itemize}
 \item $\alpha_1 \coloneqq e_2 -e_3$;
 \item $\alpha_2 \coloneqq e_3 -e_4$;
 \item $\alpha_3 \coloneqq e_4$;
 \item $\alpha_4 \coloneqq \frac{1}{2}(e_1-e_2 -e_3-e_4)$.
\end{itemize}

The Dynkin diagram associated to this root system (with the above basis) is as follows:
\[
\dF
\]
and the Cartan matrix is
\[
\begin{pmatrix}
 2&-1&0&0\\
 -1&2&-2&0\\
 0&-1&2&-1\\
 0&0&-1&2
\end{pmatrix}
\]

With respect to this choice of basis $\Delta$, the set $\Phi^{+}$ of positive roots consists of the following vectors:
$(1,0,0,0)$,
$(0,1,0,0)$,
$(0,0,1,0)$,
$(0,0,0,1)$,
$(1,1,0,0)$,
$(0,1,1,0)$,
$(0,0,1,1)$,
$(1,1,1,0)$,
$(0,1,1,1)$,
$(1,1,1,1)$,
$(0,1,2,0)$,
$(1,1,2,0)$,
$(0,1,2,1)$,
$(1,2,2,0)$,
$(1,1,2,1)$,
$(0,1,2,2)$,
$(1,2,2,1)$,
$(1,1,2,2)$,
$(1,2,3,1)$,
$(1,2,2,2)$,
$(1,2,3,2)$,
$(1,2,4,2)$,
$(1,3,4,2)$,
$(2,3,4,2)$,
where $(a,b,c,d)$ means the root $a \alpha_1 +b \alpha_2 + c \alpha_3 + d \alpha_4 $.

The fundamental weights are as follows:
\begin{align*}
 \omega_1 &\coloneqq (2,3,4,2), \\
 \omega_2 &\coloneqq (3,6,8,4), \\
 \omega_3 &\coloneqq (2,4,6,3), \\
 \omega_4 &\coloneqq (1,2,3,2), \\
\end{align*}
and the half-sum of positive roots is
\[
\rho = \omega_1 + \omega_2 + \omega_3 + \omega_4 = (8,15,21,11).
\]

In the following we denote by $H_\omega$ the divisor corresponding to a weight $\omega$.
Then, by \cite[Theorem~24.10]{BH59}, the degree of $H_\omega$ on $E$ is as follows:
\[
\deg H_{\omega} = (\dim  E)! \prod_{\gamma \in \sC} \frac{(\omega,\gamma)}{(\rho, \gamma)},
\]
where the set $\sC$ of complementary roots to the subset $\{\alpha _1, \alpha_3\}$ is, by definition, 
the set of positive roots each of which is not a linear combination of roots in $\Delta\setminus\{\alpha _1, \alpha_3\}$.
By the description of positive roots, we see that $\sC$  is $\Phi^{+}\setminus\{\alpha_2, \alpha_4\}$, which consists of $22$ vectors.
Therefore
\begin{align*}
 \xi (Z) &= 23 \int_{0}^{8}(3-x)(x p_Y^{*}H_Y + (8-x) p_Z^*H_Z)^{22} dx\\
 &= 23 \int_{0}^{8}(3-x) (22 !) \prod_{\gamma \in \sC}\frac{(x \omega_1+ (8-x) \omega_3,\gamma)}{(\rho,\gamma)} dx.
\end{align*}

Note that the matrix of products $(\omega_i, \alpha_j)$ are as follows:
\[
\begin{pmatrix}
 1&0&0&0\\
 0&1&0&0\\
 0&0&1/2&0\\
 0&0&0&1/2
\end{pmatrix}.
\]
Thus the products with complementary roots are given as follows:
\begin{longtable}{c|c|c}
 & $(x\omega_1+(8-x)\omega_3,\blank)$ & $(\rho,\blank)$ \\ \hline
 $(1,0,0,0)$ & $x$ & $1$ \\
 $(0,0,1,0)$ & $(8-x)/2$ & $1/2$ \\
 $(1,1,0,0)$ & $x$ & $2$\\
 $(0,1,1,0)$ & $(8-x)/2$ & $3/2$\\
\hline
$(0,0,1,1)$ & $(8-x)/2$ & $1$\\
$(1,1,1,0)$ & $(8+x)/2$ & $5/2$\\
$(0,1,1,1)$ & $(8-x)/2$ & $2$\\
$(1,1,1,1)$ & $(8+x)/2$ & $3$\\
\hline
$(0,1,2,0)$ & $8-x$ & $2$\\
$(1,1,2,0)$  & $8$ & $3$\\
$(0,1,2,1)$  & $8-x$ & $5/2$\\
$(1,2,2,0)$ & $8$ & $4$\\
\hline
$(1,1,2,1)$ & $8$ & $7/2$\\
$(0,1,2,2)$ & $8-x$ & $3$\\
$(1,2,2,1)$ & $8$ & $9/2$\\
$(1,1,2,2)$ & $8$ & $4$\\
\hline
$(1,2,3,1)$ & $(24-x)/2$ & $5$\\
$(1,2,2,2)$ & $8$ & $5$\\
$(1,2,3,2)$ & $(24-x)/2$ & $11/2$\\
$(1,2,4,2)$ & $16-x$ & $6$\\
\hline
$(1,3,4,2)$ & $16-x$ & $7$\\
$(2,3,4,2)$ & $16$ & $8$
\end{longtable}
Thus, we have
\begin{align*}
\prod_{\gamma \in \sC}{(\rho,\gamma)} &= 
1 \cdot \frac{1}{2} \cdot 2 \cdot \frac{3}{2} \cdot 1 \cdot \frac{5}{2} \cdot 2 \cdot 3 \cdot 2 \cdot 3 \cdot \frac{5}{2} \cdot 4 \cdot \frac{7}{2} \cdot 3 \cdot \frac{9}{2} \cdot 4 \cdot 5 \cdot 5 \cdot \frac{11}{2} \cdot 6 \cdot 7 \cdot 8\\
 &= 2^4 \cdot 3^7 \cdot 5^4 \cdot 7^2 \cdot 11
\end{align*}
and
\begin{align*}
 \prod_{\gamma \in \sC}(x \omega_1+ (8-x) \omega_3,\gamma) &= 
 -2^{14}x^2(x-8)^7(x+8)^2(x-24)^2(x-16)^2.
\end{align*}
Therefore
\begin{align*}
 \xi (Z) &= \frac{23 \cdot (22 !) \cdot 2^{14}}{2^4 \cdot 3^7 \cdot 5^4 \cdot 7^2 \cdot 11} \int_{0}^{8} x^2 (x-3) (x-8)^7 (x+8)^2 (x-24)^2 (x-16)^2 dx\\
 &= 2^{73} \cdot 19 \cdot 23 \cdot 199 \cdot 1049 \\
 &>0.
\end{align*}
Hence $\PF$ is K-polystable.

\subsection{Proof of Theorem~\ref{thm:KE} for $\PG$}
First we recall the root system  $G_2$.
Consider two vectors $\alpha_1$ and $\alpha_2$ in the Euclidian vector space $V = \bR^2$, where
\begin{itemize}
 \item $\|\alpha_1\|^2 =2$;
 \item $\|\alpha_2\|^2 = 6$;
 \item $(\alpha_1,\alpha_2) = -3 $.
\end{itemize}
Then the set of the following vectors form the root system $G_2$:
$\pm\alpha_1$, $\pm \alpha_2$, $\pm (\alpha_1+\alpha_2)$, $\pm(2\alpha_1+\alpha_2)$, $\pm(3\alpha_1+\alpha_2)$, $\pm(3\alpha_1+2\alpha_2)$.
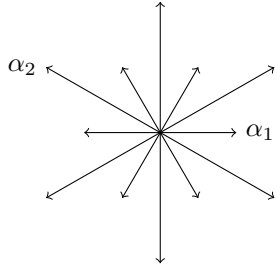
\[
\rG
\]
Then the roots $\alpha_1$ and $\alpha_2$ define a basis of $G_2$ and the Dynkin diagram is as follows:
\[
\dG
\]

Then the fundamental weights are as follows:
\begin{itemize}
 \item $\omega_1 = 2\alpha_1 + \alpha_2$; 
 \item $\omega_2 = 3\alpha_1 +2\alpha_2$,
\end{itemize}
and the half sum of positive roots is
\[
\rho = \omega_1 + \omega_2 = 5\alpha_1 +3\alpha_2.
\]

Note that the matrix of products $(\omega_i, \alpha_j)$ is
\[
\begin{pmatrix}
 1&0\\
 0&3
\end{pmatrix}.
\]

By considering the direct sum with the root system $A_1$, we have the root system of $A_1 \times G_2$.
In the following, we denote the positive root of $A_1$ by $\alpha_0$ and thus the Dynkin diagram is as follows:
\[
\dAG
\]

Then the half sum of positive roots of $A_1\times G_2$ is
\[
\rho = \omega_0 + \omega_1 + \omega_2 =\alpha_0 + 5\alpha_1 +3\alpha_2.
\]
Similarly to the case $X = \PF$, we have
\begin{align*}
 \xi (Z) &= 8 \int_{0}^{3} (2-x) (x p_Y^{*} H_Y + (6-2x) p_Z^* H_Z)^{7} dx\\
 &= 8 \int_{0}^{3} (2-x) (7 !) \prod_{\gamma \in \sC} \frac{(x \omega_Y + (6-2x) \omega_Z, \gamma)}{(\rho, \gamma)} dx,
\end{align*}
where the set $\sC$ of complementary roots to the subset $\{\alpha _0, \alpha_1, \alpha_2\}$ is, by definition, the set of positive roots each of which is not a linear combination of roots in $\Delta\setminus\{\alpha_0, \alpha_1,\alpha_2\} = \emptyset$.
Thus $\sC$ is the set of positive roots $\Phi^{+}$, which consists of $7$ vectors.

Note that $\omega_Y = \omega_2$ and $\omega_Z = \omega_0 +\omega_1$.
Thus the products with complementary roots are given as follows:
\begin{longtable}{c|c|c}
 & $(x\omega_Y+(6-2x)\omega_Z,\blank)$ & $(\rho,\blank)$ \\ \hline
 $\alpha_0$ & $6-2x$ & $1$ \\
 $\alpha_1$ & $6-2x$ & $1$ \\
 $\alpha_2$ & $3x$ & $3$\\
 $\alpha_1+\alpha_2$ & $6+x$ & $4$\\
 $2\alpha_1+\alpha_2$ & $12-x$ & $5$\\
 $3\alpha_1+\alpha_2$ & $18-3x$ & $6$\\
 $3\alpha_1+2\alpha_2$ & $18$ & $9$
\end{longtable}

Therefore
\begin{align*}
\prod_{\gamma \in \sC}{(\rho,\gamma)} &= 1 \cdot 1 \cdot 3 \cdot 4 \cdot 5 \cdot 6 \cdot 9 \\
&= 2^3 \cdot 3^4 \cdot 5
\end{align*}
and
\begin{align*}
\prod_{\gamma \in \sC}(x \omega_Y+ (6-2x) \omega_Z,\gamma) &= 2^{3}\cdot3^4\cdot x(x-3)^2(x-6)(x+6)(x-12).
\end{align*}
Hence we have
\begin{align*}
 \xi (Z)  &= \frac{-8 \cdot (7 !) \cdot 2^{3} \cdot 3^4}{2^3 \cdot 3^4 \cdot 5} \int_{0}^{3} x (x-2) (x-3)^2 (x-6) (x+6) (x-12) dx\\
 &= 2^{4} \cdot 3^9 \cdot 5 \cdot 11 \\
 &>0.
\end{align*}
Then $\PG$ is K-polystable.

\section{Remarks}\label{sect:remark}
\subsection{Remark on specialization of $\PG$}
As is mentioned in \cite[Remark 4.2]{Kan20}, $X= \PG$ is a Mukai $8$-fold of genus $7$, i.e.\ $-K_X=6H_X$ and $H_X^8=12$.
By \cite{Kuz18}, it is known that there are two isomorphic classes $X_{\gen}$ and $X_{\spe}$ of Mukai $8$-fold of genus $7$.
Moreover $X_{\gen}$ degenerates isotrivially to $X_{\spe}$.

The following proposition can be found in \cite[Proposition 4.8, Remark 4.10, Remark 5.4]{BFM20}:

\begin{proposition}
$\PG \simeq X_{\gen}$.
\end{proposition}

\begin{proof}
Set $X = \PG$.
Note that $\Aut^0(X) = A_1 \times G_2$ and thus
\[
\dim H^0 (T_X) = 17.
\]
On the other hand we know that $\chi(T_X) = 17$.
By the Akizuki-Nakano vanishing, we have
\[
\dim H^{\geq 2} (T_X) = 0.
\]
Thus $\dim H^1(T_X) =0$ and thus $X$ is rigid.
\end{proof}

\begin{proof}[Proof of Corollary~\ref{cor:K-unstable}]
Assume to that contrary that $X_{\spe}$ is K-semistable.
Then, $X_{\spe}$ degenerates to a K-polystable Fano variety \cite{CDS15c}, \cite{LWX18}, and we obtain a contradiction from the description of K-moduli \cite{BX19}, \cite{LWX19}, \cite{LWX18}, \cite{Oda15}, \cite{SSY16}.
For instance, by \cite[Theorem~1.1]{BX19}, $X_{\gen}$ and $X_{\spe}$ are $S$-equivalent.
By \cite{LWX18}, $X_{\gen}$ and $X_{\spe}$ have a common K-polystable degeneration, which should be $X_{\gen}$.
This contradicts to the fact that $X_{\gen}$ is rigid.
\end{proof}

\begin{remark}
 In fact, by carefully checking constructions in \cite{Kuz18}, we have a test configuration which degenerates $X_{\gen}$ to $X_{\spe}$.
Together with the K-polystability of $X_{\gen}$, this implies the K-unstability of $X_{\spe}$.
\end{remark}

\begin{remark}\label{rem:nonreductive}
 In a response to the first version of this paper, Baohua Fu has pointed out that the automorphism group of $X_{\spe}$ is not reductive as follows:
By \cite[Remark 2.13]{FH18} (see also \cite[Section 4.2]{BFM20}), $X_{\spe}$ is an equivariant compactification of $\bC^8$.
Then, it follows from \cite[Proposition 1]{AP14} that the automorphism group of $X_{\spe}$ is not reductive, since $X_{\spe}$ is not homogeneous.

This implies that there are no K\"ahler-Einstein metrics on $X_{\spe}$.
Note that the non-reductivity of the automorphism group is not enough to conclude the K-unstability of $X_{\spe}$.
\end{remark}

\subsection{Remark on foliations}
Let $X$ be a Fano manifold satisfying Condition~\ref{cond:Pasquier}.
Then $X$ is horospherical or isomorphic to $\PF$ or $\PG$.
In each case, the blow-up along the closed orbit admits a similar diagram as in \eqref{eq:diagram}, and the smooth morphism $\pi$ defines a $G$-invariant foliation $\sF$ on $X$, which is called the \emph{canonical foliation} on $X$ \cite[Section 3]{Kan20}.
The first Chern class of $\sF$ and the existence of K\"ahler-Einstein metrics are summarized as follows (see also \cite[Proposition 3.5]{Kan20}):
\begin{longtable}{c|c|c}
$X$ & $c_1(\sF)$ & KE or not \\ \hline
horospherical & $>0$ & Not KE \\
$\PG$, $\PF$ & $0$ & KE
\end{longtable}
A naive question asks whether or not there is a relation between $G$-invariant foliations with positive first Chern class and K\"ahler-Einstein metrics.

\bibliographystyle{amsalpha}
\bibliography{}
\end{document}